\documentclass[12pt,reqno]{amsart}
\usepackage{amsmath}
\usepackage{amsthm}
\usepackage{amssymb}
\usepackage{graphics}
\usepackage{latexsym}

\numberwithin{equation}{section}
\newtheorem{thm}{Theorem}[section]
\newtheorem{prop}[thm]{Proposition}
\newtheorem{lem}[thm]{Lemma}
\newtheorem{cor}[thm]{Corollary}

\theoremstyle{remark}
\newtheorem{rem}{Remark}[section]
\newtheorem{defn}{Definition}

\newcommand{\BBB}{\mathbb}
\newcommand{\R}{{\BBB R}}
\newcommand{\Z}{{\BBB Z}}
\newcommand{\T}{{\BBB T}}
\newcommand{\Zd}{\dot{\BBB Z}}
\newcommand{\N}{{\BBB N}}
\newcommand{\C}{{\BBB C}}
\newcommand{\LR}[1]{{\langle {#1} \rangle }}
\newcommand{\lec}{{\ \lesssim \ }}
\newcommand{\gec}{{\ \gtrsim \ }}

\newcommand{\cross}{\times}

\newcommand{\al}{\alpha}
\newcommand{\be}{\beta}
\newcommand{\ga}{\gamma}

\newcommand{\si}{\sigma}
\newcommand{\vp}{\varphi}

\newcommand{\e}{\varepsilon}
\newcommand{\ta}{\tau}
\newcommand{\x}{\xi}

\newcommand{\p}{\partial}
\newcommand{\la}{\lambda}
\newcommand{\La}{\Lambda}
\renewcommand{\th}{\theta}

\newcommand{\I}{\infty}

\newcommand{\D}{{\mathcal D}}

\newcommand{\EQ}[1]{\begin{equation} \begin{split} #1
 \end{split} \end{equation}}
\newcommand{\EQS}[1]{\begin{align} #1 \end{align}}
\newcommand{\EQQS}[1]{\begin{align*} #1 \end{align*}}
\newcommand{\EQQ}[1]{\begin{equation*} \begin{split} #1
 \end{split} \end{equation*}}

\newcommand{\Sp}{\mathcal{S'}}

\newcommand{\1}{{\bf 1}}

\newcommand{\ti}{\widetilde}
\newcommand{\ha}{\widehat}

\newcommand{\Hd}{\dot{H}}
 
\title[KdV with quasi periodic initial data]{Local well-posedness of the KdV equation with quasi periodic initial data
}

\author[K. Tsugawa]{Kotaro Tsugawa
}
\address{Graduate School of Mathematics, Nagoya University,
Chikusa-ku, Nagoya, 464-8602, Japan}
\email[K. Tsugawa]{tsugawa@math.nagoya-u.ac.jp}


\begin{document}

\begin{abstract}
We prove the local well-posedness for the Cauchy problem of the Korteweg-de Vries equation in a quasi periodic function space.
The function space contains functions such that $f=f_1+f_2+...+f_N$ where $f_j$ is in the Sobolev space of order $s>-1/2N$ of $2\pi\al^{-1}_j$ periodic functions.
Note that $f$ is not a periodic function when the ratio of periods $\al_i/\al_j$ is irrational. The main tool of the proof is the Fourier restriction norm method introduced by Bourgain.
We also prove an ill-posedness result in the sense that the flow map (if it exists) is not $C^2$,
which is related to the Diophantine problem.
\end{abstract}
\maketitle


\section{Introduction}

We consider the Cauchy problem of the Korteweg-de Vries equation as follows:
\EQ{\label{KDV}
u_t +u_{xxx}+(u^2)_x=0, \qquad u(0)=f.
}
where $u(x,t)$ and $f(x)$ are complex or real valued functions.
The local and global well-posedness for \eqref{KDV} in $H^s(\T)$ has been intensively studied.
In the present paper, we show the local well-posedness in a quasi periodic function space.
Note that a $2\pi$-periodic function on $\R$ can be written as $f=\sum_{j=1}^\infty a_je^{i\la_j x}$ with $\{\la_j\}_{j=1}^\infty\subset \Z$.
On the other hand, when $f$ is an almost periodic function, the distribution of $\{\la_j\}_{j=1}^\infty$ on $\R$ can be totally arbitrary.
When $f$ is a quasi periodic function, there exist $N\in\N$, $\al\in \R^N$ such that $\{\la_j\}_{j=1}^\infty\subset \{\al\cdot k: k \in\Z^N\}$.
First, we recall some known results of the well-posedness for \eqref{KDV} in $H^s(\T)$.
In \cite{Bo}, Bourgain proved the local and global well-posedness with $s \ge 0$ by introducing the Fourier restriction norm method.
In \cite{Ke}, Kenig, Ponce and Vega proved refined bilinear estimates to extend Bourgain's local well-posedness result to $s \ge -1/2$.
In \cite{Co}, Colliander, Keel, Staffilani, Takaoka and Tao proved the global well-posedness with $s \ge -1/2$.
Since the KdV equation on $H^s$ with $s<0$ has no conservation law, it seemed difficult to consider the long time behavior of solutions
in $H^s$ with $s<0$.
To overcome this difficulty, they introduced a regularizing Fourier multiplier operator $I$
and calculated a modified energy defined in $H^s$ with $s\ge -1/2$, which is called the ``$I$-method''.
Kappeler and Topalov proved the global well-posedness with $s\ge -1$ by the inverse scattering method in \cite{Kap}.
It is known that the flow map is analytic if and only if $s\ge -1/2$.

For the well-posedness results in $H^s(\R)$, see \cite{Bn}, \cite{Ka}, \cite{Gi}, \cite{Bo}, \cite{Ke}, \cite{Co}, \cite{Ki1} and \cite{G}.
For the ill-posedness results and counter examples of bilinear estimates related to this problem, see \cite{Ke}, \cite{Tz}, \cite{Na}, \cite{Ke2}, \cite{Chr}, \cite{Ki2} and \cite{Mol}.
To the best of author's knowledge, there are few results for almost periodic initial data.
In \cite{Eg}, Egorova proved that global solutions of \eqref{KDV} exist when the initial data $f$ is real valued and $f \in Q_s(A)$ for $s\in \N$ and $s \ge 4$ and that the solutions are almost periodic in $t$ and $x$.
Here, we say $f \in Q_s(A)$ if and only if 
\EQQ{
\|f\|_s=\sup_{x} \Big(\int_x^{x+1} |f|^2+|\p_x^s f|^2 \, dx \Big)^{1/2} <\infty
}
and there exists a sequence of $a_n$-periodic functions $f_n \in H^s(\R/a_n)$ such that
the following condition holds:
\EQ{
\forall p>0, \qquad\qquad \lim_{n\to\infty} e^{pa_{n+1}}\|f-f_n\|_{s}=0,
}
where $A=\{a_n\}$ is an increasing sequence of positive numbers such that $a_{n+1}/a_{n}\in \N \setminus \{1\}$.
Note that $f_1+f_2 \notin Q_s(A)$ if $f_j$ is an $a_j$-periodic function for $j=1,2$ and $a_1/a_2$ is irrational.
Namely, quasi periodic functions are not in $Q_s(A)$.
The Navier-Stokes equations with almost periodic initial data were studied in \cite{Giga1}, \cite{Giga2}, \cite{Yone} and \cite{Giga3}.
Kenig, Ponce and Vega studied the KdV with unbounded initial data in \cite{Ke1}, which does not include almost periodic initial data.

In the proof of the local well-posedness results in $H^s(\T)$ (\cite{Bo}, \cite{Ke}, \cite{Co}), the following argument is crucial.
If $u$ is a solution of \eqref{KDV} on $\T$, we have the following conservation law:
\EQ{
\int_{\T} u(t)\,dx =\int_{\T} f \, dx=c \label{eee1}
}
for any $t$.
Put
\EQ{\label{t1}
v=u-c.
}
Then, $v$ satisfies
\EQ{\label{KDV'}
v_t +2cv_x +v_{xxx}+(v^2)_x=0, \qquad v(0)=f-c
}
and $\int_{\T} v(t) \, dx=\int_{\T} v(0) \, dx =0$.
In the proof of the local well-posedness by the Fourier restriction norm method,
the presence of the term $2cv_x$ in \eqref{KDV'} does not make any difference.
Therefore, it is enough to study \eqref{KDV} under the assumption $\int_{\T} u(t) \, dx = \int_{\T} f \, dx=0$ when $f \in H^s(\T)$.
When $u$ is a spatially quasi periodic function, we have the following conservation law instead of \eqref{eee1}:
\EQQ{
\lim_{M\to \infty}\frac{1}{2M}\int_{-M}^{M} u(t)\,dx =\lim_{M\to \infty}\frac{1}{2M}\int_{-M}^{M} f \, dx=c
}
for any $t$
and we can use the same argument as in the case of $\T$ above.
Therefore, we assume $ \lim_{M\to \infty}\frac{1}{2M}\int_{-M}^{M} u(t) \, dx = \lim_{M\to \infty}\frac{1}{2M}\int_{-M}^{M} f \, dx=0$
in the present paper.

We assume that $N \in \N$, $\al=(\al_1,\al_2,\cdots,\al_N)\in \R^N$ satisfy $\al\cdot k \neq 0$
for any $k\in \Zd^N$ where $\Zd^N=\Z^N \setminus (0,\cdots,0)$ throughout the paper.
When $f_j$ are $2\pi\al_j^{-1}$ periodic functions, $\sum_{j=1}^N f_j$
is not a periodic function but a quasi periodic function.
It is a natural question to ask whether the local well-posedness holds or not in a function space
which contains $\sum_{j=1}^N f_j$ with $f_j\in \Hd^{\si_j}(\R/2\pi\al_j^{-1}\Z)$.
A quasi periodic function $\sum_{j=1}^N f_j$ can be written as follows:
\EQQ{
\sum_{j=1}^N\sum_{k\in \Zd} \ha{f}_j(k) \exp (i\al_j k x) 
}
 where $\ha{f}_j(k) : \Zd \to \C$ are the Fourier coefficients of $f_j$.
Since the KdV equation has a quadratic term, we need to
consider multiplications of functions of the form above.
Therefore, we introduce the following function space.
\begin{defn}
For $\si=(\si_1,\cdots,\si_N)\in \R^N$ and  $a\in \R$, we define the Banach space $G^{\si,a}$ by
\EQ{\label{Gdef}
G^{\si,a}=\Big\{\, f(x)=\sum_{k\in \Zd^N} \ha{f}(k)\exp (i\al\cdot k x)  \,\Big|\, \ha{f}: \Zd^N \to \C, \, \|f\|_{{G}^{\si,a}} <\I \,\Big\}
}
where
\EQQ{
\|f\|_{G^{\si,a}}:=\|\ha{f}\|_{\ha{G}^{\si,a}}:=\Big\|\, |\al\cdot k|^a \prod_{j=1}^N \LR{k_j}^{\si_j} \ha{f}(k) \,\Big\|_{l^2(\Zd^N)}, \qquad
\LR{\cdot}=(1+|\cdot|^2)^{1/2}.
}
\end{defn}
Note that the series in \eqref{Gdef} does not converge even in $\Sp$ when $\si_j$ are small (see Lemma \ref{ill2}).
Therefore, we introduce the following notations and assumption.
\begin{defn}
Let $e_j$ be the standard unit vector in $\R^N$ whose $j$-th component is equal to $1$ and the others are equal to $0$.
Put $d_j=\frac{s}{N-1}\{(1,1,\ldots,1)-e_j\}$.
For $s\in \R$, we define
\EQ{
\La_s=\Big\{ \sum_{j=1}^N \th_j d_j \in \R^N \, \Big|\, \sum_{j=1}^N \th_j =1, \, \th_j \ge 0 \,\, (1\le j \le N) \Big\}.
}
\end{defn}
\noindent
Using the fact that $\La_s$ is the set of internal dividing points of $d_j (1\le j \le N)$,
we can easily check that the series in \eqref{Gdef} converges in $\Sp$ when $a\le 0$ and the following assumption holds:
\EQQ{
\hspace*{-10em}\text{(A)} \hspace*{11em} s>(N-1)/2, \qquad \qquad \si \in \La_s.
}
So, we assume (A) throughout the paper (except Lemma \ref{ill2}).
Here, we give some remarks on (A), $G^{\si,a}$ and $\La_s$.
\begin{rem}\label{R1}
If $f_j \in \Hd^{\si_j+a}(\R/2\pi\al_j^{-1}\Z)$, then $\sum_{j=1}^N f_j \in G^{\si,a}$.
\end{rem}
\begin{rem}
If $\si \in \La_s$, then $\sum_{j=1}^N \si_j =s$.
\end{rem}
\begin{rem}
For simpleness we assume $\si_1 \le \si_2 \le \cdots \le \si_N$.
Then, (A) is equivalent to that $\si_1\ge 0$ and $\si_1+\cdots +\si_j >(j-1)/2$
for any $j$ satisfying $2 \le j \le N$.
\end{rem}

We rewrite \eqref{KDV} into the following integral equation:
\EQ{\label{KDVI}
u(t)=e^{-t\p_x^3} f+ \int_0^t e^{-(t-t')\p_x^3} \p_x(u^2(t')) \, dt'.
}
Now, we state our main result.
\begin{thm}\label{maintheorem}
Assume {\upshape (A)}.
Then, \eqref{KDVI} is locally well-posed in $G^{\si,-1/2}$.
\end{thm}
\begin{rem}
If $\si_1=\cdots=\si_N$, then (A) is equivalent to $\si_j >1/2-1/2N$.
Therefore, from Remark \ref{R1}, it follows that
if $r>-1/2N$ and $f_j\in \Hd^r(\R/2\pi\al_j^{-1}\Z)$ for $1\le j \le N$, then $\sum_{j=1}^N f_j \in G^{\si,-1/2}$ with (A).
Note that when $N=1$, the lower bound $-1/2N$ coincides $-1/2$, which is the critical value to assure that the flow map in $H^s(\T)$ is analytic.
In this sense, Theorem \ref{maintheorem} is a generalization of the local well-posedness result for the KdV on $\T$ by Kenig, Ponce and Vega in \cite{Ke}.
\end{rem}
For a more precise statement of this theorem, see Proposition \ref{mainprop}.
The main idea in the present paper is the definition of $G^{\si,a}$ and to take $a=-1/2$.
It is known that a nonlinear interaction between high and very low frequency parts of solutions make the well-posedness problem difficult in the study of the KdV equation.
To avoid this difficulty, we usually use the transform \eqref{t1} for the periodic case.
It is not enough for the quasi periodic case because the support of $\mathcal{F}_x[f](\xi)=\sum_{k\in\Zd^N} \delta(\xi-\al\cdot k)\ha{f}(k)$ can be dense around $\xi=0$ when $f$ is quasi periodic.
In our proof, it is crucial to assume $\mathcal{F}_x[f](\xi)$ is very small around $\xi=0$.
This is the reason why we use the homogeneous weight function $|\al\cdot k|^a$ in the definition of $G^{\si,a}$ and take $a=-1/2$ (see Remark \ref{rem10}).
Similar ideas were used for some different problems, for instance, for a stochastic KdV equation in \cite{de2}, for quadratic nonlinear Schr\"odinger equations in \cite{KT} and for the fifth order KdV equation in \cite{Ktaka}.

We also prove an ill-posedness result in the sense that the flow map in $G^{\si,a}$ (if it exists)
is not $C^2$ at the origin
under some condition on $\si$ and $a$, which is related to the Diophantine problem (see Proposition \ref{ill1}).
A relation between a Diophantine condition and the well-posedness of a coupled system of KdV equations
is studied by Oh in \cite{Oh1} and \cite{Oh2}.
The method used for the ill-posedness result is introduced by Bourgain in \cite{Bo2} to prove that the flow map of KdV in $H^s(\T)$ (if it exists) with $s<-1/2$
can not be $C^3$.

If $u$ is a real valued solution of \eqref{KDV} and $f$ is a quasi periodic function, the following conservation law holds (at least formally):
\EQQ{
\|u(t)\|_{G^{0,0}}^2=\lim_{M\to \infty}\frac{1}{2M}\int_{-M}^{M} u^2(t) \,dx=\lim_{M\to \infty} \frac{1}{2M}\int_{-M}^{M} f^2 \,dx
=\|f\|_{G^{0,0}}^2.
}
Unfortunately, we can not extend the local solutions obtained in Theorem \ref{maintheorem} to global ones
by using this conservation law because $\|u(t)\|_{G^{\si,-1/2}}$ with the assumption (A) can not be bounded by $\|u(t)\|_{G^{0,0}}$.
On the other hand, for the periodic case, we have the global well-posedness in $H^s(\T)$ with $s\ge -1/2$ in \cite{Co}. 
Therefore, we can easily obtain the following result by combining the global well-posedness in $H^{\si_1-1/2}(\R/2\pi\al_1^{-1}\Z)$
and Theorem \ref{maintheorem}.
\begin{cor}\label{cor}
Assume {\upshape (A)}.
Let $r>0$ and $T>0$.
Then there exists $\e=\e(r,T)>0$ which satisfies the following:
if initial data $f\in G^{\si,-1/2}$ satisfies $f=g+h$, $g\in B_r(H^{\si_1-1/2}(\R/2\pi\al_1^{-1}\Z))$, $h\in B_\e(G^{\si,-1/2})$ and
$g$ is real valued,
then the solution of \eqref{KDVI} obtained in Theorem \ref{maintheorem} can be extended on $t \in [-T ,T]$.  
\end{cor}

The main tool of the proof of Theorem \ref{maintheorem} is the Fourier restriction norm method.
We define the Fourier restriction norm as follows:
\EQQ{
Z^{\si,a}&=\{ u(t) \in \Sp(\R: {G}^{\si,a})  \,|\, \|{u}\|_{{Z}^{\si,a}}<\I \},\\
\|u\|_{Z^{\si,a}}&=\|{u}\|_{{X}^{\si,a,1/2}}+\|{u}\|_{{Y}^{\si,a,0}},\\
\|u\|_{X^{\si,a,b}}&=\|\LR{\ta+(\al\cdot k)^3}^b \ti{u}\|_{\ha{G}^{\si,a} L^2_\ta},\\
\|u\|_{Y^{\si,a,b}}&=\|\LR{\ta+(\al\cdot k)^3}^b \ti{u}\|_{\ha{G}^{\si,a} L^1_\ta},
}
where $\ti{u}=\mathcal{F}_t \ha{u}$.
Note that ${Z}^{\si,a}$ is continuously embedded in $C_t(\R : {G}^{\si,a})$.

In the proof of Theorem \ref{maintheorem}, the following bilinear estimates play an important role.
\begin{prop}\label{MainEst}
Assume {\upshape (A)}.
Let $T>0$.
Then, there exists $\e>1/8$ such that
the following estimates hold for any $u,v$ satisfying $\mathrm{supp}\, u, \mathrm{supp}\, v \subset [-T,T]\cross\R$:
\EQS{
\| (uv)_x \|_{X^{\si,-1/2,-1/2}} \lec T^\e \|u\|_{X^{\si,-1/2,1/2}}\|v\|_{X^{\si,-1/2,1/2}},\label{e41}\\
\| (uv)_x \|_{Y^{\si,-1/2,-1}} \lec T^\e \|u\|_{X^{\si,-1/2,1/2}}\|v\|_{X^{\si,-1/2,1/2}}.\label{e42}
}
\end{prop}

Finally, we give some notations.
We will use $A\lesssim B$ to denote an estimate of the form $A \le CB$ for some constant $C$ and write $A \sim B$ to mean $A \lesssim B$ and $B \lesssim A$.
Implicit constants may depend on $\si, b, \al$ and $N$.
For a Banach space $X$, We define $B_r(X)=\{ f\in X \,|\, \|f\|_X \le r \}$.
Let $\chi(t)$ be a smooth function such that $\chi(t)=1$ for $|t|<1$ and $\chi(t)=0$ for $|t|>2$ and $\chi_T(t)=\chi(t/T)$.
$\mathcal{F}_t$ (resp. $\mathcal{F}_x$) is the Fourier transform with respect to $t$ (resp. $x$).
The rest of this paper is planned as follows.
In Section 2, we give some notations and preliminary lemmas.
In Section 3, we prove bilinear estimates.
In Section 4, we prove Theorem \ref{maintheorem} and Corollary \ref{cor}.
In Section 5, we prove an ill-posedness result.

\section*{Acknowledgement}
A part of this paper was written while the author was visiting the University of Toronto.
He would like to thank James Colliander and the Department of Mathematics at the University of Toronto
for their hospitality and his helpful comments.
He also would like to thank Tadahiro Oh for his helpful comments.
This research was supported by KAKENHI (22740088) and by FY 2011 Researcher Exchange Program between JSPS and NSERC.


\section{preliminary lemmas}

\begin{lem}\label{LE3}
Assume {\upshape (A)}.
Let $a \in \R$, $\e >\e'>0$ and  $T>0$.
Then, it follows that
\EQ{
\| u \|_{X^{\si,a,1/2-\e}} \lec T^{\e'}\|u\|_{X^{\si,a,1/2}}
}
for any $u$ satisfying $\mathrm{supp}\, u \subset [-T,T]\cross \R$.
\end{lem}
\begin{proof}
Let $\mathrm{supp}\, g(t) \subset [-T,T]$.
Then, by the Plancherel theorem and the triangle inequality, we have
\EQQ{
 \|g\|_{H^{1/2-\e}} = \|\chi_T g\|_{H^{1/2-\e}}
  \lec \|(\LR{\tau}^{1/2-\e}|\mathcal{F}\chi_T|)*|\mathcal{F}g| \|_{L^2_\tau}+\| |\mathcal{F}\chi_T|*(\LR{\tau}^{1/2-\e}|\mathcal{F}g|) \|_{L^2_\tau}
}
By the Young inequality and the direct calculation, the first term is bounded by
\EQQ{
\|\LR{\tau}^{1/2-\e} \mathcal{F}\chi_T \|_{L_\tau^p} \|\mathcal{F}g \|_{L_\tau^q} 
  \lec T^{\e'}\|\LR{\tau}^{1/2}\mathcal{F}g \|_{L_\tau^2},
}
where $1/p=\e-\e'+1/2$ and $1/q=1-\e+\e'$.
In the same manner, the second term is bounded by
\EQQ{
\| \mathcal{F}\chi_T \|_{L_\tau^p} \|\LR{\tau}^{1/2-\e}\mathcal{F}g \|_{L_\tau^q}
  \lec T^{\e'}\|\LR{\tau}^{1/2}\mathcal{F}g \|_{L_\tau^2},
}
where $1/p=1-\e'$ and $1/q=\e'+1/2$.
Therefore, we obtain
\EQQ{
\| u \|_{X^{\si,a,1/2-\e}} &=\|\LR{\p_t}^{1/2-\e} \exp(-t\p_x^3) u \|_{G^{\si,a}L^2_t}\\
&\lec T^{\e'}\|\LR{\p_t}^{1/2} \exp(-t\p_x^3) u \|_{G^{\si,a}L^2_t}=T^{\e'}\|u\|_{X^{\si,a,1/2}}.
}
\end{proof}
Lemmas \ref{LE4}, \ref{LE40} are variants of the Sobolev inequality.
\begin{lem}\label{LE4}
(i) Let $\min_{1 \le j \le N} \{ \si_j\} >1/2$.
Then, it follows that
\EQ{
\|uv\|_{G^{\si,0}}\lec \|u\|_{G^{\si,0}}\|v\|_{G^{\si,0}}.\label{e39}
}
(ii) Let $\si_1=0, \min_{2 \le j \le N} \{ \si_j\} >1/2$ and $a>1/4$.
Then, it follows that
\EQ{
\Big\|\sum_{k'\in \Zd^N,k'\neq k}\LR{|\al\cdot k||\al\cdot(k-k')||\al\cdot k'|}^{-a} |\ha{u}(k-k')| |\ha{v}(k')|\Big\|_{\ha{G}^{\si,0}}\lec \|u\|_{G^{\si,0}}\|v\|_{G^{\si,0}}.\label{e40}
}
\end{lem}
\begin{proof}
By the Schwarz inequality, we have
\EQQ{
\|uv\|^2_{G^{\si,0}} &\leq \sum_{k\in\Zd^N}\prod_{j=1}^N\LR{k_j}^{2\si_j}\Big|\sum_{k'\in\Zd^N,\, k'\neq k} \ha{u}(k-k')\ha{v}(k') \Big|^2\\
&\leq C_1(\si) \sum_{k\in\Zd^N}\sum_{k'\in\Zd^N,\, k'\neq k} \prod_{j=1}^N\LR{k_j-k_j'}^{2\si_j}\LR{k_j'}^{2\si_j} \Big|\ha{u}(k-k')\ha{v}(k') \Big|^2\\
&\leq C_1(\si) \|u\|^2_{G^{\si,0}}\|v\|^2_{G^{\si,0}},
}
where
\EQQ{
C_1(\si)=\sup_{k\in\Zd^N} \sum_{k'\in\Zd^N,\, k'\neq k} \prod_{j=1}^N\LR{k_j}^{2\si_j}\LR{k_j-k_j'}^{-2\si_j}\LR{k_j'}^{-2\si_j}.
}
Since $C_1(\si)< \infty$, we obtain \eqref{e39}.
Put
\EQQ{
F(k)&=\sum_{k'\in \Zd^N,\, k'\neq k}\LR{|\al\cdot k||\al\cdot(k-k')||\al\cdot k'|}^{-a} |\ha{u}(k-k')| |\ha{v}(k')|,\\
F_1(k)&=F(k)\big|_{|\al\cdot k|\ge 1},\qquad F_2(k)=F(k)\big|_{|\al\cdot k|<1}.
}
In the same manner as the proof of \eqref{e39}, we obtain
\EQ{
\| F_1 \|_{\ha{G}^{\si,0}}\lec \|u\|_{G^{\si,0}}\|v\|_{G^{\si,0}}\label{e80}
}
since
\EQQ{
&C_2(\si,a)\\
=&\sup_{|\al\cdot k|\ge 1} \sum_{k'\in\Zd^N,\, k'\neq k}  \LR{|\al\cdot k||\al\cdot(k-k')||\al\cdot k'|}^{-2a}\prod_{j=2}^N\LR{k_j}^{2\si_j}\LR{k_j-k_j'}^{-2\si_j}\LR{k_j'}^{-2\si_j}\\
\lec& \sum_{k'\in\Zd^N,\, k'\neq k}  \LR{|\al\cdot(k-k')||\al\cdot k'|}^{-2a}
  \prod_{j=2}^N \min \{ \LR{k_j-k_j'}, \LR{k_j'} \}^{-2\si_j}< \infty.
}
Because $\#\{ k_1\in\Z \,|\, |\al\cdot k|<1 \}\sim 1$ for any $(k_2,\cdots,k_N)\in \Z^{N-1}$, we have
\EQ{\label{ee1}
\| F_2 \|^2_{\ha{G}^{\si,0}}&\lec \sum_{k \in \Zd^N} \prod_{j=2}^N \LR{k_j}^{2\si_j}\Big(\sum_{k'\in\Zd^N,\, k'\neq k} |\ha{u}(k-k')||\ha{v}(k')| \Big)^2 \Big|_{|\al\cdot k|<1}\\
&\lec \sup_{k_1 \in \Z} \sum_{(k_2,\cdots,k_N) \in \Z^{N-1}} \prod_{j=2}^N \LR{k_j}^{2\si_j}\Big(\sum_{k'\in\Zd^N,\, k'\neq k} |\ha{u}(k-k')||\ha{v}(k')| \Big)^2.
}
By the Schwarz inequality, we have
\EQ{\label{ee2}
&\Big( \sum_{k'\in\Z^N,\, k'\neq k} |\ha{u}(k-k')||\ha{v}(k')|\Big)^2\\
\lec& \Big(\sum_{(k'_2,\cdots, k'_N)\in \Z^{N-1}} \|\ha{u}(k-k')\|_{l^2_{k'_1}}\|\ha{v}(k')\|_{l^2_{k'_1}}\Big)^2\\
\lec& \sum_{(k'_2,\cdots, k'_N)\in \Z^{N-1}}  \prod_{j=2}^N \LR{k_j-k'_j}^{-2\si}\LR{k_j'}^{-2\si_j}\\
  &\times  \sum_{(k'_2,\cdots, k'_N)\in \Z^{N-1}} \Big( \prod_{j=2}^N \LR{k_j-k'_j}^{2\si}\LR{k_j'}^{2\si_j} \Big) \|\ha{u}(k-k')\|^2_{l^2_{k'_1}} \|\ha{v}(k')\|^2_{l^2_{k'_1}}.
}
Here we have
\EQ{\label{ee3}
\sup_{(k_2,\cdots,k_N) \in \Z^{N-1}}\sum_{(k'_2,\cdots, k'_N)\in Z^{N-1}} \prod_{j=2}^N \LR{k_j}^{2\si_j}\LR{k_j-k'_j}^{-2\si_j}\LR{k'_j}^{-2\si_j} \lec 1.
}
Therefore, from \eqref{ee1}--\eqref{ee3}, we obtain
\EQ{\label{e81}
&\| F_2 \|^2_{\ha{G}^{\si,0}}\\
\lec &\sum_{(k_2,\cdots,k_N) \in \Z^{N-1}}
\sum_{(k'_2,\cdots, k'_N)\in \Z^{N-1}}  \Big( \prod_{j=2}^N \LR{k_j-k'_j}^{2\si_j}\LR{k_j'}^{2\si_j}\Big)\|\ha{u}(k-k')\|^2_{l^2_{k'_1}} \|\ha{v}(k')\|^2_{l^2_{k'_1}}\\
\lec &\|u\|^2_{G^{\si,0}}\|v\|^2_{G^{\si,0}}.
}
Collecting \eqref{e80} and \eqref{e81}, we obtain \eqref{e40}.
\end{proof}
\begin{lem}\label{LE40}
Let $A, B, C\in \R, \min\{a,b,c\} \ge 0, \min\{a+b, b+c, c+a \}> 1/2$ and $a+b+c>1$.
Then, it follows that
\EQ{
\|\LR{\cdot +A}^{-a} f*g\|_{L^1} \lec \|\LR{\cdot +B}^b f\|_{L^2}\|\LR{\cdot +C}^c g\|_{L^2}\label{e99}
}
where the implicit constant depends only on $a,b$ and $c$.
\end{lem}
\begin{proof}
\eqref{e99} is equivalent to
\EQQ{
\| \int \LR{\ta+A}^{-a}\LR{\ta-\ta'+B}^{-b}f(\ta-\ta') \LR{\ta'+C}^{-c}g(\ta') \,d\ta' \|_{L^1} \lec \| f\|_{L^2}\| g\|_{L^2}
}
By using the Schwarz inequality, the left hand-side is bounded by
\EQQ{
\| \LR{\ta+A}^{-a}\LR{\ta-\ta'+B}^{-b} \LR{\ta'+C}^{-c}\|_{L^2_{\ta}L^2_{\ta'}} \| f(\ta-\ta')g(\ta')\|_{L^2_{\ta}L^2_{\ta'}}
\lec  \| f\|_{L^2}\| g\|_{L^2}.
}
Here, we used
\EQQ{
\int\int \LR{\ta+A}^{-2a}\LR{\ta-\ta'+B}^{-2b}\LR{\ta'+C}^{-2c}\,d\ta'd \ta \le C
}
where the constant depends only on $a,b$ and $c$.
\end{proof}

\begin{lem}\label{LE5}
Let $\ta, \ta' \in \R$, $k, k' \in \Z^N$.
Then, it follows that
\EQ{
M:&=\max \{ |\ta+(\al\cdot k)^3|, |\ta-\ta'+(\al\cdot (k-k'))^3|, |\ta'+(\al\cdot k')^3|\}\\
  &\gec |\al\cdot k||\al\cdot (k-k')||\al\cdot k'|.
}
\end{lem}
The proof of Lemma \ref{LE5} follows from the triangle inequality and $a^3-(a-b)^3-b^3=3ab(a-b)$.
The following lemmas are standard tools of the Fourier restriction norm.
For the proofs, see e.g. \cite{Gi2}.
\begin{lem}\label{LE1}
Assume {\upshape (A)}.
Let $a\in \R$.
Then, for any $f$, it follows that
\EQQ{
\|\chi(t) e^{ -t\p_x^3} f\|_{Z^{\si,a}}\lec \|f\|_{G^{\si,a}}.
}
\end{lem}

\begin{lem}\label{LE2}
Assume {\upshape (A)}.
Let $a\in \R$.
Then, for any  $F$, it follows that 
\EQQ{
\| \chi(t)\int_0^t e^{-(t-t')\p_x^3} F(t') \, dt' \|_{Z^{\si,a}}
  \lec \|F\|_{X^{\si,a,-1/2}}+\|F\|_{Y^{\si,a,-1}}.
}
\end{lem}


\section{proof of the bilinear estimates}
Our aim in this section is to prove Proposition \ref{MainEst}.
The proof of the following lemma is almost similar to that of Theorem 1.2 in \cite{Ke}.
The difference is the presence of $P, Q$.
It is crucial that the implicit constant in \eqref{e28} does not depend on $P, Q$ to apply Lemma \ref{MainLem} for the proof of Lemma \ref{BiEstLem}.

\begin{lem}\label{MainLem}
Let $\al_1 \neq 0$ and $b>3/8$.
Then, the following estimate holds for any $P, Q \in \R$ and
any functions $\ti{u}(\tau, k_1), \ti{v}(\tau, k_1)$ on $\R \cross \Z$:
\EQ{\label{e28}
&\| |\ti{u}|*_{\tau,k_1}|\ti{v}| \|_{\ell^2_{k_1}L^2_\tau}\\
\lec &\|\LR{\tau+(\al_1 k_1+P)^3}^b\ti{u}(\tau,k_1)\|_{\ell^2_{k_1}L^2_\tau}
  \|\LR{\tau+(\al_1 k_1+Q)^3}^b\ti{u}(\tau,k_1)\|_{\ell^2_{k_1}L^2_\tau},
}
where the implicit constant depends only on $\al_1$ and $b$.
\end{lem}
\begin{proof}
We first consider the region $A=\big\{k_1 \in \Z \,\big|\, 3\al_1^2|\al_1k_1+P+Q|>1  \big\}$.
Applying the Schwartz inequality, we have
\EQQ{
|\ti{u}| *_{\tau,k_1}|\ti{v}|
\lec I(\tau,k_1)^{1/2} \big( \LR{\tau+(\al_1 k_1+P)^3}^{2b}|\ti{u}|^2 *_{\tau,k_1} \LR{\tau+(\al_1 k_1+Q)^3}^{2b} |\ti{v}|^2 \big)^{1/2}
}
where
\EQQ{
I(\tau,k_1) = \LR{\tau+(\al_1 k_1+P)^3}^{-2b} *_{\tau,k_1}\LR{\tau+(\al_1 k_1+Q)^3}^{-2b}.
}
If we have
\EQ{\label{e1}
\sup_{\tau\in \R, k_1\in A} I(\tau,k_1) \lec 1,
}
then we obtain
\EQQ{
&\| |\ti{u}|*_{\tau,k_1}|\ti{v}| \|_{\ell^2_{k_1}(A)L^2_\tau}\\
\lec &\| (\LR{\tau+(\al_1 k_1+P)^3}^{2b}|\ti{u}|^2)*_{\tau,k_1}
  ( \LR{\tau+(\al_1 k_1+Q)^3}^{2b}|\ti{v}|^2) \|^{1/2}_{\ell^1_{k_1}L^1_\tau}\\
\lec &\|\LR{\tau+(\al_1 k_1+P)^3}^b\ti{u}\|_{\ell^2_{k_1}L^2_\tau}
  \|\LR{\tau+(\al_1 k_1+Q)^3}^b\ti{v}\|_{\ell^2_{k_1}L^2_\tau}
}
from the Young inequality.
Therefore, we have only to check \eqref{e1}.
Since $\LR{\tau+\be}^{-2b}*_\tau \LR{\tau+\ga}^{-2b }\lec \LR{\tau+\be+\ga}^{\min\{1-4b, -2b\}}$,
we have
\EQQ{
  I(\tau,k_1)\lec \sum_{k_1'\in\Z}\LR{\tau+(\al_1 k'_1+Q)^3+(\al_1 (k_1-k'_1)+P)^3}^{\min\{1-4b, -2b\}}.
}
Note that 
\EQQ{
 &\tau+(\al_1 k'_1+Q)^3+(\al_1 (k_1-k'_1)+P)^3\\
=&\ta+3\al_1^2(\al_1k_1+P+Q)k_1'^2+\Gamma(\al_1,k_1,k'_1,P,Q),
}
and the order of $\Gamma (\al_1,k_1,k'_1,P,Q)$ with respect to $k_1'$ is less than $2$.
Because $\min\{1-4b, -2b\} < -1/2$, we obtain \eqref{e1}.

Next, we consider the region $A^c=\{k_1 \in \Z \,|\, 3\al_1^2|\al_1k_1+P+Q| \le 1  \}$.
Since $\# A^c \sim \al_1^{-3}$, it follows that
\EQ{
\| |\ti{u}|*_{\tau,k_1}|\ti{v}| \|_{\ell^2_{k_1}(A^c)L^2_\tau}
\lec & \al_1^{-3} \sup_{k_1\in\Z} \Big\| \sum_{k_1'} |\ti{u}(\tau,k_1-k'_1)|*_\tau |\ti{v}(\tau,k_1')| \Big\|_{L^2_\tau}\\
\lec & \al_1^{-3} \sup_{k_1\in\Z} \sum_{k_1'} \Big\| |\ti{u}(\tau,k_1-k'_1)|*_\tau |\ti{v}(\tau,k_1')| \Big\|_{L^2_\tau}.\label{e50}
}
Applying the Schwartz inequality, we have
\EQ{
&|\ti{u}(\tau,k_1-k'_1)|*_{\tau} |\ti{v}(\tau,k'_1)|\\
\lec &J(\tau,k,k'_1 )^{1/2} \Big((\LR{\tau+ ( \al_1 (k_1-k'_1 )+P)^3}^{2b}|\ti{u}(\tau,k_1-k'_1)|^2)\\
 &\hspace{13em} *_\tau
  (\LR{\tau+(\al_1 k'_1+Q)^3}^{2b} |\ti{v}(\tau,k'_1)|^2)\Big)^{1/2}\label{e51}
}
where
\EQQ{
J(\tau,k_1,k'_1) = \LR{\tau+(\al_1 (k_1-k'_1)+P)^3}^{-2b} *_\tau
  \LR{\tau+(\al_1 k'_1+Q)^3}^{-2b}.
}
Obviously,
\EQ{
\sup_{\tau\in \R, k_1 \in \Z, k'_1\in A^c} J(\tau,k_1,k'_1) \lec 1.\label{e52}
}
Therefore, from \eqref{e50}--\eqref{e52}, we obtain
\EQQ{
&\| |\ti{u}|*_{\tau,k_1}|\ti{v}| \|_{\ell^2_{k_1}(A^c)L^2_\tau}\\
\lec & \al_1^{-3} \sup_{k_1\in\Z} \sum_{k_1'} \Big\|\LR{\tau+(\al_1 (k_1-k'_1)+P)^3}^{2b}|\ti{u}(\tau,k_1-k'_1)|^2\\
  &\hspace*{17em}*_\tau \LR{\tau+(\al_1 k'_1+Q)^3}^{2b}|\ti{v}(\tau,k_1')|^2 \Big\|^{1/2}_{L^1_\tau}\\
\lec & \al_1^{-3} \sup_{k_1\in\Z} \sum_{k_1'} \Big\|\LR{\tau+(\al_1 (k_1-k'_1)+P)^3}^{b}|\ti{u}(\tau,k_1-k'_1)| \Big\|_{L^2_\tau}\\
  &\hspace*{17em}\cross \Big\|\LR{\tau+(\al_1 k'_1+Q)^3}^{b}|\ti{v}(\tau,k_1')| \Big\|_{L^2_\tau}\\
\lec &\al_1^{-3}\|\LR{\tau+(\al_1 k_1+P)^3}^b\ti{u}(\tau,k_1)\|_{\ell^2_{k_1}L^2_\tau}
  \|\LR{\tau+(\al_1 k_1+Q)^3}^b\ti{u}(\tau,k_1)\|_{\ell^2_{k_1}L^2_\tau}.
}
\end{proof}

\begin{lem}\label{BiEstLem}
Assume {\upshape (A)} and $b>3/8$.
Then, the following bilinear estimates hold:
\EQS{
\| uv \|_{X^{\si,0,0}} \lec \|u\|_{X^{\si,0,b}}\|v\|_{X^{\si,0,b}},\label{BE1}\\
\| uv \|_{X^{\si,0,-b}} \lec \|u\|_{X^{\si,0,b}}\|v\|_{X^{\si,0,0}},\label{BE2}\\
\| uv \|_{Y^{\si,0,-1/2}} \lec \|u\|_{X^{\si,0,b}}\|v\|_{X^{\si,0,b}},\label{BE3}
}
\end{lem}
\begin{proof}
Since $\si$ is an internal dividing point of $sd_j/(N-1)$,
we have the estimates for general $\si$ by interpolating among the estimates for $\si=sd_j/(N-1)$ with $j=1,\cdots,N$.
By symmetry, we have only to consider the case $\si=sd_1/(N-1)$, namely, $\si_1=0, \si_2=\cdots=\si_N=s/(N-1)>1/2$.

We first prove \eqref{BE1}.
The left hand-side of \eqref{BE1} is bounded by
\EQQ{
&\Big\| \prod_{j=2}^N \LR{k_j}^{\si_j} \sum_{k'\in\Zd^{N},\,k'\neq k}\int |\ti{u}(\ta-\ta',k-k')||\ti{v}(\ta',k')| d\ta'\Big\|_{\ell_k^2L_\tau^2} \lec\\
&\Big\| \prod_{j=2}^N \LR{k_j}^{\si_j} \sum_{k_2'\cdots k_N'} \Big\| \sum_{k_1'}\int |\ti{u}(\ta-\ta',k-k')||\ti{v}(\ta',k')| d\ta'\Big\|_{\ell_{k_1}^2L_\tau^2} \Big\|_{\ell^2_{k_2\cdots k_N}}.
}
Applying Lemma \ref{MainLem} with $P=\sum_{j=2}^N \al_j(k_j-k_j')$, $Q=\sum_{j=2}^N \al_jk_j$ and \eqref{e39} with respect to $k_2,\cdots,k_N$ variables,
we obtain \eqref{BE1}.

By the duality argument, \eqref{BE2} is equivalent to
\EQQ{
\| uv \|_{X^{-\si,0,0}} \lec \|u\|_{X^{\si,0,b}}\|v\|_{X^{-\si,0,b}},\\
}
which is obtained in the same manner as the proof of \eqref{BE1} if we apply the duality of \eqref{e39}:
\EQQ{
\|uv\|_{G^{-\si,0}}\lec \|u\|_{G^{\si,0}}\|v\|_{G^{-\si,0}}
}
instead of \eqref{e39}.

Finally, we prove \eqref{BE3}.
Divide $\R\times\Zd^N\times\R\times\Zd^N$ into the following three parts:
\EQQ{
\Omega_1&=\big\{(\ta,k,\ta',k') \,\big|\, |\ta+(\al\cdot k)^3|
  \ge \max \{ |\ta'+(\al\cdot k')^3)|, |\ta-\ta'+(\al\cdot(k-k'))^3|\} \big\} \,\\
\Omega_2&=\big\{(\ta,k,\ta',k') \,\big|\, |\ta'+(\al\cdot k')^3|
  \ge \max \{ |\ta+(\al\cdot k)^3)|, |\ta-\ta'+(\al\cdot(k-k'))^3|\} \big\},\\
\Omega_3&=\big\{(\ta,k,\ta',k') \,\big|\, |\ta-\ta'+(\al\cdot (k-k'))^3|
  \ge \max \{ |\ta'+(\al\cdot k')^3)|, |\ta+(\al\cdot k)^3|\} \big\}.
}
Put
\EQ{\label{DE}
\ti{F}_j(u,v)=
  \sum_{k'\in \Zd^N,\,k'\neq k}\int|\ti{u}(\ta-\ta',k-k')||\ti{v}(\ta',k')|\1_{\Omega_j}\,d\ta'
}
for $j=1,2,3$ where
\EQQ{
\1_{\Omega}(\ta,k,\ta',k')=
\begin{cases}
1 \,\,\,\, \text{when}\,\, (\ta,k,\ta',k')\in \Omega,\\
0 \,\,\,\, \text{otherwise}.
\end{cases}
}
Then, we have
\EQQ{
\| uv \|_{Y^{\si,0,-1/2}} \lec \sum_{j=1}^3 \| F_j(u,v) \|_{Y^{\si,0,-1/2}}.
}
By symmetry, we have only to estimate $\| F_j(u,v) \|_{Y^{\si,0,-1/2}}$ for $j=1,2$.

In $\Omega_1$, it follows that $|\ta+(\al\cdot k)^3| \gec |\al\cdot k||\al\cdot(k-k')||\al\cdot k'|$
from Lemma \ref{LE5}.
Therefore, we have
\EQQ{
\| F_1(u,v) \|_{Y^{\si,0,-1/2}}
\lec \Big\|\LR{\ta+(\al\cdot k)^3}^{b-5/8}\sum_{k'}\int 
 \LR{|\al\cdot k||\al\cdot(k-k')||\al\cdot k'|}^{1/8-b}\\
\times |\ti{u}(\ta-\ta',k-k')| |\ti{v}(\ta',k')| \,d\ta' \Big\|_{\ha{G}^{\si,0}L^1_\ta},
}
which is bounded by
\EQQ{
&\Big\| \sum_{k'} \LR{|\al\cdot k||\al\cdot(k-k')||\al\cdot k'|}^{1/8-b}\\
&\hspace*{4em}\times \|\LR{\ta+(\al\cdot (k-k'))^3}^b\ti{u}(\ta,k-k')\|_{L^2_\ta} \|\LR{\ta+(\al\cdot k')^3}^b\ti{v}(\ta,k')\|_{L^2_\ta}  \Big\|_{\ha{G}^{\si,0}}\\
\lec& \|u\|_{X^{\si,0,b}}\|v\|_{X^{\si,0,b}}
}
from Lemma \ref{LE40} and Lemma \ref{LE4}.

In $\Omega_2$, it follows that $|\ta'+(\al\cdot k')^3| \ge |\ta+(\al\cdot k)^3)|$.
Take $\e>0, b'>3/8$ such that $b>b'+\e$.
Then, we have
\EQQ{
\| F_3(u,v) \|_{Y^{\si,0,-1/2}}
  &\le \|\LR{\ta+(\al\cdot k)^3}^{-1/2-\e} (|\ti{u}|*_{\ta,k}\LR{\ta+(\al\cdot k)^3}^\e |\ti{v}|)\|_{\ha{G}^{\si,0}L^1_\ta}\\
  &\lec \||\ti{u}|*_{\ta,k}\LR{\ta+(\al\cdot k)^3}^\e |\ti{v}|\|_{\ha{G}^{\si,0}L^2_\ta}\\
  &\lec \|u\|_{X^{\si,0,b'}}\|v\|_{X^{\si,0,b'+\e}}.
}
Here, we used the Schwartz inequality in the second line and \eqref{BE1} in the final line.
\end{proof}

\begin{proof}[Proof of Proposition \ref{MainEst}]
First, we prove \eqref{e41}.
By symmetry and the decomposition \eqref{DE}, we have only to prove
\EQ{
\| F_j(u,v)(\ta,k) \|_{X^{\si,1/2,-1/2}} \lec T^\e \|u\|_{X^{\si,-1/2,1/2}} \|v\|_{X^{\si,-1/2,1/2}}\label{e44}
}
for $j=1, 2$.
In $\Omega_1$, it follows that $\LR{\ta+(\al\cdot k)^3}^{-1/2}|\al\cdot k|^{1/2} \lec |\al\cdot k'|^{-1/2}|\al\cdot (k-k')|^{-1/2}$.
Therefore, for any $b>3/8$, the left-hand side of \eqref{e44} with $j=1$ is bounded by
\EQQ{
\| (|\al\cdot k|^{-1/2}|\ti{u}|) * (|\al\cdot k|^{-1/2}|\ti{v}|) \|_{X^{\si,0,0}} \lec  \|u\|_{X^{\si,-1/2,b}}\|v\|_{X^{\si,-1/2,b}}
}
where we used \eqref{BE1}.
From Lemma \ref{LE3}, we obtain \eqref{e44} with $j=1$.
In $\Omega_2$, it follows that $|\al\cdot k|^{1/2}\lec \LR{\ta'+(\al\cdot k')^3}^{1/2}|\al\cdot k'|^{-1/2}|\al\cdot (k-k')|^{-1/2}$.
Therefore, for any $b>3/8$, the left-hand side of \eqref{e44} with $j=2$ is bounded by
\EQQ{
&\| (|\al\cdot k|^{-1/2}|\ti{u}|) *_{\ta,k} (|\al\cdot k|^{-1/2}|\ta'+(\al\cdot k')|^{1/2}|\ti{v}|) \|_{X^{\si,0,-1/2}}\\
\lec & \|u\|_{X^{\si,-1/2,b}}\|v\|_{X^{\si,-1/2,1/2}}
}
where we used \eqref{BE2}.
From Lemma \ref{LE3}, we obtain \eqref{e44} with $j=2$.

Next, we prove \eqref{e42}.
By symmetry and the decomposition \eqref{DE}, we have only to prove
\EQ{
\| F_j(u,v)(\ta,k) \|_{Y^{\si,1/2,-1}} \lec T^\e \|u\|_{X^{\si,-1/2,1/2}} \|v\|_{X^{\si,-1/2,1/2}}\label{e45}
}
for $j=1, 2$.
In $\Omega_1$, it follows that $\LR{\ta+(\al\cdot k)^3}^{-1/2}|\al\cdot k|^{1/2}\lec |\al\cdot k'|^{-1/2}|\al\cdot (k-k')|^{-1/2}$.
Therefore, for any $b>3/8$, the left-hand side of \eqref{e45} with $j=1$ is bounded by
\EQQ{
\| (|\al\cdot k|^{-1/2}|\ti{u}|) *_{\ta,k} (|\al\cdot k|^{-1/2}|\ti{v}|) \|_{Y^{\si,0,-1/2}}
\lec  \|u\|_{X^{\si,-1/2,b}}\|v\|_{X^{\si,-1/2,b}}
}
where we used \eqref{BE3}.
From Lemma \ref{LE3}, we obtain \eqref{e45} with $j=1$.
In $\Omega_2$, it follows that $|\al\cdot k|^{1/2}\lec \LR{\ta'+(\al\cdot k')^3}^{1/2}|\al\cdot k'|^{-1/2}|\al\cdot (k-k')|^{-1/2}$.
Therefore, for any $1/2>b>3/8$ the left-hand side of \eqref{e45} with $j=2$ is bounded by
\EQQ{
&\| (|\al\cdot k|^{-1/2}|\ti{u}|) *_{\ta,k} (|\al\cdot k|^{-1/2}|\ta+(\al\cdot k)^3|^{1/2}|\ti{v}|) \|_{Y^{\si,0,-1}}\\
\lec &\| (|\al\cdot k|^{-1/2}|\ti{u}|) *_{\ta,k} (|\al\cdot k|^{-1/2}|\ta+(\al\cdot k)^3|^{1/2}|\ti{v}|) \|_{X^{\si,0,-b}}\\
\lec & \|u\|_{X^{\si,-1/2,b}}\|v\|_{X^{\si,-1/2,1/2}}
}
where we used the Schwarz inequality with respect to $\ta$ in the second line and \eqref{BE2} in the final line.
From Lemma \ref{LE3}, we obtain \eqref{e45} with $j=2$.

\end{proof}


\section{well-posedness results}
In this section, we give outlines of the proofs of the local well-posedness result below and Corollary \ref{cor}.
\begin{prop}\label{mainprop}
Assume {\upshape (A)}.

(Existence) Let $0<\theta<1/8$.
For any $r>0$, there exists $T\sim \min \{ r^{-1/\theta},1 \}$ which satisfies the following:
for any  $f\in B_r(G^{\si,-1/2})$, there exists $u\in C([-T,T]; G^{\si,-1/2})\cap Z^{\si,-1/2}$
satisfying \eqref{KDVI} on $t\in[-T,T]$.
Moreover, the flow map, $f\in B_r(G^{\si,-1/2}) \mapsto u\in C([-T,T]; G^{\si,-1/2})\cap Z^{\si,-1/2}$
is Lipschitz continuous.

(Uniqueness) If $ u,v \in C([-T,T]; G^{\si,-1/2})\cap Z^{\si,-1/2}$ satisfy \eqref{KDVI} on $t\in[-T,T]$, then $u(t)=v(t)$ on $t\in [-T,T]$.
\end{prop}
We can prove this proposition by the standard method of the Fourier restriction norm method (see e.g. \cite{Gi2}, \cite{Ke}).
Therefore, we describe only the proof of the existence.
\begin{proof}
Put
\EQQ{
M(u)=\chi(t) e^{-t\p_x^3} f+ \chi(t)\int_0^t e^{-(t-t')\p_x^3} \p_x(\chi_T(t')u(t'))^2 \, dt'.
}
We shall show $M$ is a contraction on $B_{C_0}(Z^{\si,-1/2})$ where $C_0>0$ is to be determined later.
Then, we have a solution of $u=M(u)$ by the fixed point argument, which satisfy \eqref{KDVI} on $t\in [-T,T]$.
From Lemma \ref{LE1}, we have
\EQQ{
\|\chi(t)e^{-t\p_x^3} f\|_{Z^{\si,-1/2}}\lec \|f\|_{G^{\si,-1/2}}.
}
From Lemma \ref{LE2} and Proposition \ref{MainEst}, we have
\EQQ{
\| \chi(t) \int_0^t e^{-(t-t')\p_x^3} \p_x(\chi_T(t')u(t'))^2 \, dt' \|_{Z^{\si,-1/2}}
  \lec T^\theta\|\chi_T u\|_{Z^{\si,-1/2}}^2 \lec  T^\theta\|u\|_{Z^{\si,-1/2}}^2.
}
Combining them, we obtain
\EQQ{
\| Mu\|_{Z^{\si,-1/2}} &\le C\|f\|_{G^{\si,-1/2}}+CT^\theta\|u\|_{Z^{\si,-1/2}}^2\\
  &\le 2Cr
}
if $u\in B_{C_0}(Z^{\si,-1/2})$, $C_0=2Cr$ and $T\le (4C^2r)^{-1/\theta}$.
We conclude that $M$ is a map from $B_{C_0}(Z^{\si,-1/2})$ to itself.
In the same manner, we have
\EQQ{
\| M(u-v)\|_{Z^{\si,-1/2}} & \le CT^\theta\|u-v\|_{Z^{\si,-1/2}}\|u+v\|_{Z^{\si,-1/2}}\\
  &\le 1/2\|u-v\|_{Z^{\si,-1/2}}
}
for $u, v \in B_{C_0}(Z^{\si,-1/2})$.
\end{proof}

Here, we introduce the localized norm of $Z^{\si,a}$;
\EQQ{
\|u\|_{Z^{\si,a}_{[t_1,t_2]}} =\{\inf \|v\|_{Z^{\si,a}} \,|\, v=u \,\,\text{on} \,\, [t_1,t_2] \}.
}
\begin{proof}[Outline of the proof of Corollary \ref{cor}]
In Theorem 2 in \cite{Co}, a priori estimate and global well-posedness for \eqref{KDVI} with real valued initial data in $\Hd^{\si_1-1/2}(\R\setminus 2\pi\al_j^{-1}\Z)$ have been proved.
Therefore, there exist $C_*=C_*(T,r)$, $\delta=\delta(T,r)$ and $1/8>\theta>0$ satisfying
\EQQS{
&\sup_{t\in[-T,T]} \|v(t)\|_{G^{\si,-1/2}} = \sup_{t\in[-T,T]} \|v(t)\|_{H^{\si_1-1/2}} \le C_*,\\
&0<\delta^\theta<<C_*^{-1},\qquad  \sup_{-k+1\le j\le k} \|v\|_{Z_{[(j-1)\delta, j\delta]}^{\si,-1/2}} < C_*, \qquad k:=T/\delta\in\N
}
for the solution $v$ of \eqref{KDVI} with initial data $g$.
Since we have the local well-posedness in $G^{\si,-1/2}$ and the existence time of solutions depends only on the size of the norm of initial data,
we only need to prove the following a priori estimate: for sufficiently small $\e>0$, the solution $u$ of \eqref{KDVI} satisfies
\EQQ{
\sup_{t\in{[-T,T]}}\|u(t)\|_{G^{\si,-1/2}} \le 2C_*.
}
For that pourpose, put $w=u-v$ and we will show
\EQ{\label{apri}
\sup_{t\in{[-T,T]}}\|w(t)\|_{G^{\si,-1/2}} \le C\e < C_*
}
where $\e>$ is to be determined later.
For simpleness we consider only $0\le t \le T$.
$w$ satisfies the following integral equation:
\EQQ{
w(t)=e^{-t\p_x^3} w((j-1)\delta)+ \int_{(j-1)\delta}^{t} e^{-(t-t')\p_x^3} \p_x(w(t')(w(t')+2v(t'))) \, dt'.
}
Therefore, if we have 
\EQQ{
\|w((j-1)\delta)\|_{G^{\si,-1/2}} \le C_{j-1}\e <C_*,
}
then we obtain
\EQQ{
\sup_{(j-1)\delta\le t \le j\delta}\|w(t)\| \le C \|w\|_{Z^{\si,-1/2}_{[(j-1)\delta, j\delta]}}\le C \|w((j-1)\delta)\|_{G^{\si,-1/2}}
   < C_j\e
}
in the same manner as the proof of Proposition \ref{mainprop} because $\delta^\theta\|w+2v\|_{Z^{\si,-1/2}_{[(j-1)\delta, j\delta]}}<<1$.
We can iterate this argument until $j=k$ if we take $\e$ sufficiently small such that $\sup _{1\le j \le k }C_{j}\e <<C_*$.
\end{proof}


\section{ill-posedness results}
 
Let $\mu\in\R, \rho\in\R$ and $K>0$.
We say that $\mu$ is a number of type $(K,\rho)$ if for any integers $p$ and $q\neq 0$,
\EQQ{
\Big|\mu-\frac{p}{q}\Big| \ge \frac{K}{|q|^{2+\rho}}.
}
Note that almost every $\mu \in\R$ is a number of type $(K,\rho)$ for some $K>0$ when $\rho>0$ by Lemma 3 on p. 114 of \cite{Ar}.
By Dirichlet's Theorem (see, e.g. Theorem 1A on p. 34 of \cite{Sc}), if $\mu\in \R$ is a number of type $(K,\rho)$
for some $K>0$, then $\rho \ge 0$.
We define $\rho_\mu$ by 
\EQQ{
\rho_\mu = \inf\,\{\rho \ge 0 \,|\, \mu \, \text{is a number of type }\, (K,\rho) \,\text{for some}\, K>0 \}
}
as in \cite{Oh1}.

For simplicity, we consider only the case $N=2$.
\begin{prop}\label{ill1}
Assume that $\mu=\al_1/\al_2$, $a > (2\min\{ \si_1, \si_2\}-\min\{1, \rho_\mu\})/(1+\rho_\mu)$ and $\min\{\si_1,\si_2\}+a \ge -2$.
Let $r>0, 1>T>0$.
Then, the flow map of \eqref{KDVI} from initial data $f \in B_r(G^{\si,a}) $ to a solution $u \in C([-T,T]:G^{\si,a})$  (if it exists)
cannot be $C^2$ differentiable at the origin.
\end{prop}
\begin{rem}\label{rem10}
When $\si_2=0, \rho_\mu=1$, the assumption on $a$ in Proposition \ref{ill1} is $a>-1/2$.
This is the reason why we take $a=-1/2$ in Theorem \ref{maintheorem}.
\end{rem}

\begin{proof}
By symmetry, we assume $\si_1 \le \si_2$.
We assume that the flow map is $C^2$ and $u(\ga,t,x)$ is the solution of \eqref{KDVI} with initial data $\ga f$
where $\ga>0$.
Then, following the argument in \cite{Tz}  (see also \cite{Bo2}), we have
\EQQ{
\frac{\p u}{\p\ga}(0,t,x)&=e^{ -t\p_x^3} f(x):=u_1(t,x),\\
\frac{\p^2 u}{\p\ga^2}(0,t,x)&=\int_0^t e^{ -(t-t')\p_x^3} \p_x (u_1(t',x))^2 \,dt':=u_2(t,x).
}
We only need to show that 
\EQQ{
\| u_2(t,x) \|_{G^{\si,a}} \lec \|f\|^2_{G^{\si,a}}
}
fails.
From the definition of $\rho_\mu$, for any $\e>0$, we have the following:\\
(i) There exists $K>0$ such that
\EQ{
\Big|\frac{\al_1}{\al_2} -\frac{p}{q}\Big| \ge \frac{K}{|q|^{2+\rho_\mu+\e}}
}
holds for any integers $p$ and $q\neq 0$.\\
(ii) For any $n\in\N$,
there exists $(p_n,q_n) \in \Z^2$ such that
\EQ{
\Big|\frac{\al_1}{\al_2} -\frac{p_n}{q_n}\Big| \le \frac{n^{-1}}{|q_n|^{2+\rho_\mu-\e}}.\label{e30}
}
Note that, from \eqref{e30}, it follows that $p_n \sim q_n \to \infty$ as $n\to \infty$ if $\e$ is sufficiently small.
Put
\EQQ{
\ha{f}_n(k)=
|q_n|^{-a-\si_1}\1_{(-q_n,0)}(k)+|\al_1q_n-\al_2p_n|^{-a}|q_n|^{-\si_1}|p_n|^{-\si_2}\1_{(q_n,-p_n)}(k)
}
where $\1_{(a,b)}(k)=1$ for $k=(a,b)$ and $\1_{(a,b)}(k)=0$ otherwise.

Then,
\EQQ{
\|f_n\|^2_{G^{\si,a}}&\sim 1.
}
We shall show that $\|u_2(t,x)\|_{G^{\si,a}}\to \infty$ as $n \to \infty$.
Put $\vp_\al(k,k')=3(\al\cdot k)(\al\cdot k')(\al\cdot (k-k'))$.
Then,
\EQQ{
\ha{e^{ t\p_x^3} u_2}(t,k)&=i\int_0^t\sum_{k'\in\Zd^N} e^{-it'\vp_\al(k,k')} (\al\cdot k)\ha{f_n}(k-k')\ha{f_n}(k') \,dt'\\
&= I_1 +I_2+ I_3
}
where
\EQQ{
I_1&= -2i\int_0^t e^{-it'\vp_\al((0,-p_n),(q_n,-p_n))} \,dt'\\
    &\hspace{8em}\cross(\al_2p_n)|\al_1q_n-\al_2p_n|^{-a}|q_n|^{-a-2\si_1}|p_n|^{-\si_2}\1_{(0,-p_n)}(k),\\
I_2&= -2i\int_0^t e^{-it'\vp_\al((-2q_n,0),(-q_n,0))} \,dt'
    (\al_1q_n)|q_n|^{-2a-2\si_1}\1_{(-2q_n,0)}(k),\\
I_3&= 2i\int_0^t e^{-it'\vp_\al((2q_n,-2p_n),(q_n,-p_n))} \,dt'\\
   &\hspace*{8em}\cross  (\al_1q_n-\al_2p_n)|\al_1q_n-\al_2p_n|^{-2a}|q_n|^{-2\si_1}|p_n|^{-2\si_2}\1_{(2q_n,-2p_n)}(k).
}
We shall use the following inequality to estimate the integral with respect to $t'$
\EQQ{
\frac{1}{\LR{y}}\gec \sup_{t\in[-T,T]}\Big| \int_0^t e^{-it'y} \, dt' \Big| \gec \frac{T}{\LR{y}}.
}
First, we estimate $\sup_{t\in[-T,T]} \|I_1\|_{G^{\si,a}}$.
Since $|q_n|^{-\rho_\mu-1+\e}\gec |\al_1q_n-\al_2p_n|\gec |q_n|^{-\rho_\mu-1-\e}$,
we have
\EQQ{
\sup_{t\in[-T,T]} \|I_1\|_{G^{\si,a}} & \gec T \frac{|\al_1q_n-\al_2p_n|^{-a}|q_n|^{-a-2\si_1}|p_n|^{a+1}}{\LR{|\al_1q_n-\al_2p_n||q_n||p_n|}}\\
  & \gec T\min\{ |q_n|^{a(\rho_\mu+1 \pm \e)-2\si_1+1}, |q_n|^{(a+1)(\rho_\mu+1 \pm \e)-2\si_1-1}\},
}
which goes to $\infty$ for sufficiently small $\e>0$.

Obviously, we have
\EQQ{
  \sup_{t\in[-T,T]} \|I_2\|_{G^{\si,a}} \lec \frac{|q_n|^{-a-\si_1+1}}{\LR{|q_n|^3}}\sim |q_n|^{-a-\si_1-2}\lec 1.
}
 
Finally, we estimate $\sup_{t\in[-T,T]} \|I_3\|_{G^{\si,a}}$.
\EQQ{
\sup_{t\in[-T,T]} \|I_3\|_{G^{\si,a}} &\lec \frac{|\al_1q_n-\al_2p_n|^{-a+1}|q_n|^{-\si_1}|p_n|^{-\si_2}}{\LR{|\al_1q_n-\al_2p_n|^3}}\\
  & << \sup_{t\in[-T,T]} \|I_1\|_{G^{\si,a}}.
}
Therefore, we conclude $\|u_2(t,x)\|_{G^{\si,a}}\to \infty$.
\end{proof}

Finally we prove the following lemma to show that the condition $s>(N-1)/2$ in (A) is necessary to assure that the series in \eqref{Gdef} converges in $\Sp$.
\begin{lem}\label{ill2}
Let $a \in \R$, $\si_j\ge 0$ for all $0\le j \le N$ and $\sum_{j=1}^N \si_j  \le (N-1)/2 $.
Then there exists $\{f_n\} \subset G^{\si,a}$ which satisfies the following:\\
(i) $\sup_{n}\|f_n\|_{G^{\si,a}} \lec 1$\\
(ii) for some $\vp \in \D$, $\LR{f_n, \mathcal{F}^{-1}_\xi[\vp]}=\LR{\mathcal{F}_x[f_n], \vp} \to \infty$ as $n\to\infty$ 
\end{lem}
\begin{proof}
Put
\EQQ{
A_{n}&=\{x \in \R^N \,|\, \LR{x} \le n \,\, \text{and}\,\, 2|\al_N| \le |\al\cdot x|\le 4|\al_N|\},\\
\ha{f_n}(k)&=
\begin{cases}
\LR{k}^{1-N}(\log \LR{k})^{-1}, \quad \text{when}\quad k\in \Z^N \cap A_{n}\\
0, \qquad \qquad \qquad  \qquad \text{otherwise}.
\end{cases}\\
}
Let $\vp(\x) \in \D$ such that $\vp(\x)=1$ for $2|\al_N|\le \x \le 4|\al_N|$
and $\vp(\x)=0$ for $\x \le |\al_N|$ or $\x \ge 5|\al_N|$.
We have the following relation between the Fourier transform and the Fourier coefficient
\EQQ{
\mathcal{F}_x[f_n](\xi)=\sum_{k\in \Zd^N} \delta(\x-\al\cdot k)\ha{f}_n(k).
}
Therefore, we have
\EQQ{
\LR{\mathcal{F}_x[f_n], \vp}\sim \sum_{k\in \Zd^N\cap A_{n}} \LR{k}^{1-N}(\log \LR{k})^{-1} \to \infty,
}
and
\EQQ{
\|f_n\|^2_{G^{\si,a}} &\sim |\al_N|^{2a}\sum_{k\in \Zd^N\cap A_{n}} \LR{k}^{2(1-N)}(\log \LR{k})^{-2} \prod_{j=1}^{N} \LR{k_j}^{2\si_j}\\
&\lec \sum_{k\in \Zd^N \cap A_{N}} \LR{k}^{1-N}(\log \LR{k})^{-2} \lec 1.
}
Here we used
\EQQ{
\lim_{n\to \infty} \int_{A_{n}} \frac{1}{\LR{x}^{N-1}\log \LR{x}} \, dx \sim C(\al_N) \int_{\R^{N-1}} \frac{1}{\LR{x}^{N-1}\log \LR{x}} \, dx =\infty
}
and
\EQQ{
\sup_n \int_{A_{n}} \frac{1}{\LR{x}^{N-1}(\log \LR{x})^2} \, dx \sim C(\al_N) \int_{\R^{N-1}} \frac{1}{\LR{x}^{N-1}(\log \LR{x})^2} \, dx \lec 1.
}
\end{proof}


\end{document}